\documentclass[a4paper,11pt, leqno]{article}

\usepackage{luatex85}
\usepackage[english]{babel}
\usepackage[utf8x]{inputenc}
\usepackage[autostyle, english = american]{csquotes}
\MakeOuterQuote{"}
\usepackage{amssymb, mathrsfs, amsthm, amsmath, amstext, url}
\usepackage{stmaryrd}
\usepackage[all]{xy}
\usepackage{enumerate}
\usepackage{mdwlist}
\usepackage{graphicx}

\usepackage[multiple]{footmisc}
\usepackage{mathtools}
\usepackage{stackrel} 
\usepackage{cancel} 
\usepackage{tikz}
\usepackage{tikz-cd}
\usepackage{float}
\usetikzlibrary{babel,cd,decorations.pathreplacing,arrows,arrows.meta}
\usepackage{enumitem}
\tikzcdset{arrow style=tikz, diagrams={>={stealth[width=4pt,length=5pt]}}}
\tikzset{>=stealth}
\tikzcdset{arrow style=tikz, diagrams={>={Computer Modern Rightarrow[width=4pt,length=5pt]}}}
\usepackage{microtype}
\usepackage{bm}
\usepackage{hyperref}
\usepackage[capitalize,nameinlink]{cleveref}

\newskip\procskipamount
\newskip\interskipamount
\newskip\refskipamount
\newcommand{\procskip}{\vskip\procskipamount}
\newcommand{\interskip}{\vskip\interskipamount}

\newcommand{\procbreak}{\par
	\ifdim\lastskip<\procskipamount\removelastskip
	\penalty-100
	\procskip\fi
	\noindent\ignorespaces}

\newcommand{\titlebreak}{\par%
	\ifdim\lastskip<\interskipamount\removelastskip%
	\penalty10000%
	\interskip\fi%
	\noindent}%

\newcommand{\interbreak}{\par%
	\ifdim\lastskip<\interskipamount\removelastskip%
	\penalty-100%
	\interskip\fi%
	\noindent\ignorespaces}%

\theoremstyle{plain}
\newtheorem{theorem}{Theorem}
\crefname{theorem}{Theorem}{Theorems}
\newtheorem{lemma}[theorem]{Lemma}
\crefname{lemma}{Lemma}{Lemmas}
\newtheorem{def+prop}[theorem]{Definition and Proposition}
\crefname{def+prop}{Definition and Proposition}{Definitions and Propositions}
\newtheorem{prop}[theorem]{Proposition}
\crefname{prop}{Proposition}{Propositions}
\newtheorem{corollary}[theorem]{Corollary}
\crefname{corollary}{Corollary}{Corollaries}

\crefname{figure}{Figure}{Figures}
\crefname{chapter}{chapter}{chapters}

\theoremstyle{definition}
\newtheorem{definition}[theorem]{Definition}
\crefname{definition}{Definition}{Definitions}
\newtheorem{def+rem}[theorem]{Definition and Remark}
\crefname{def+rem}{Definition and Remark}{Definitions and Remarks}
\newtheorem{remark}[theorem]{Remark}
\crefname{remark}{Remark}{Remarks}
\newtheorem{notation}[theorem]{Notation}
\crefname{notation}{Notation}{Notations}

\theoremstyle{remark}

\crefname{example}{Example}{Examples}

\DeclareMathOperator{\Hom}{Hom}
\DeclareMathOperator{\colim}{colim}
\DeclareMathOperator{\card}{card}
\DeclareMathOperator{\spec}{Spec}

\DeclareMathOperator{\orbit}{orbit}
\DeclareMathOperator{\norm}{N}
\DeclareMathOperator{\charpol}{Pc}
\DeclareMathOperator{\trace}{tr}

\newcommand*{\doi}[1]{\href{http://dx.doi.org/#1}{\textsc{doi}: #1}}

\makeatletter
\newcommand{\textlabel}[2]{%
	\protected@edef\@currentlabel{#1}
	\phantomsection
	#1\label{#2}
}
\makeatother

\title{Inheritance of local topological properties of schemes}

\author{Johann Gramzow}
\date{ }

\begin{document}

	\maketitle

    \begin{abstract}
		\noindent This paper  shows that for certain local topological properties, given a locally quasi-finite, flat and locally finitely presented map of schemes \(f\colon X\to Y\), if \(Y\) has the property, then so does \(X\). We also show that being locally topologically Noetherian or locally connected or having locally finitely many irreducible components are examples of these local topological properties.
	\end{abstract}
    
	\section*{Introduction}
	
	\noindent In this paper, we consider the permanence of local topological properties under preimages. More precisely, we consider the question of which conditions must be imposed on a morphism of schemes \(f\colon X\to Y\) such that if \(Y\) has a property \(\textbf{P}\), then \(X\) also has \(\textbf{P}\). Ofer Gabber investigated this for \(\textbf{P}=\text{"locally connected"}\) in \cite{gabber}, while Peter Scholze and Bhargav Bhatt proved that \(\textbf{P}=\text{"locally topologically Noetherian"}\) is stable under taking étale preimages in \cite{Bhatt Scholze}. We note and prove in \cref{section 1} that for these properties we have the following implications:
	\begin{align*}
		\substack{ \text{locally}\\\text{Noetherian} }\Rightarrow\substack{ \text{locally topologically}\\\text{Noetherian} }\Rightarrow \substack{ \text{locally finitely many}\\\text{ irreducible components} }\Rightarrow \substack{ \text{locally}\\\text{ connected.} }
	\end{align*}
	\noindent Recently, these weaker topological properties have become more important in homotopy theory of schemes. Classically, the étale fundamental group of a scheme \(X\) is defined only for locally Noetherian schemes via (the automorphism group of) the fiber functor on a covering space. Even refined invariants like the étale homotopy type by Artin-Mazur (cf. \cite{AM69}) or the étale topological type by Friedlander (cf. \cite{FL82}) use this restrictive assumption on \(X\) although the constructions work more generally for all locally connected schemes. A more recent refinement of the étale fundamental group, the proétale fundamental group by Bhatt and Scholze introduced in \cite{Bhatt Scholze} and defined via so-called geometric coverings, relaxed the condition posed on \(X\) to locally topologically Noetherian. 
	However, Remark 7.3.11 in \cite{Bhatt Scholze} indicates that the definition works in the generality of all schemes having locally finitely many irreducible components. We strengthen  this claim by proving the counterpart of \cite[Lemma 6.6.10 (3)]{Bhatt Scholze} for this weaker property, which states that it lifts along étale morphisms and thus to all geometric coverings. The overarching aim of this paper is to give sufficient nontrivial conditions for permanence of certain local topological properties and showing that \(\textbf{P}=\) "locally connected", \(\textbf{P}=\) "locally finitely many irreducible components" and \(\textbf{P}=\) "locally topologically Noetherian" fulfill these conditions. For this, we follow the ideas of Gabber in \cite{gabber}. More exactly: 
	
	\vspace*{0.2cm}
	\noindent Let \(\textbf{P}\) be a property of schemes and let the following four conditions be denoted by \textlabel{$\dagger$}{intro:dagger}: 
	\begin{itemize} 
		\item Let \(f\colon X\to Y\) be an integral surjective morphism of affine schemes, and $G$ the constant \(Y\)-group scheme of a finite group acting on \(X\) via \(\rho \colon G\times_YX\to X\), such that \(\rho\times_Yp_X\colon G\times_Y X\to X\times_YX\) is surjective. Then, if \(Y\) has \(\textbf{P}\), so does \(X\).
		\item Let \(A,B\) be rings. If \(\spec{\left(B\right)}\to \spec{\left(A\right)}\) is a surjective morphism of schemes and \(\spec{\left(B\right)}\) has \(\textbf{P}\), then so does \(\spec{\left(A\right)}\).
		\item If a scheme \(X\) has \(\textbf{P}\), then so does every open subscheme \(U\subseteq X\).
		\item If a scheme \(X\) has an open covering \((X_i)_{i\in I}\) such that each \(X_i\) has \(\textbf{P}\), then \(X\) has \(\textbf{P}\).
	\end{itemize} 
	
	\begin{prop}[\cref{cor: main 1,cor: main 2,cor: main 3}]
		The properties "locally connected", "locally finitely many irreducible components" and "locally topologically Noetherian" fulfill \ref{intro:dagger}.
	\end{prop}
	
	\begin{theorem}[\cref{prop: main 4}]
		Let $f\colon X \to Y$ be a  flat, locally quasi-finite morphism of schemes locally of finite presentation, \(\textbf{P}\) a property that satisfies \ref{intro:dagger} and let $Y$ have \(\textbf{P}\). Then, $X$ has \(\textbf{P}\). 
	\end{theorem}
	
	\noindent As étale morphisms are also flat, locally of finite presentation and locally quasi-finite we in particular find:
	
	\begin{corollary}
		Let \(f\colon X\to Y\) be an étale morphism of schemes, \(\textbf{P}\) a property that satisfies \ref{intro:dagger} and let $Y$ have \(\textbf{P}\). Then, $X$ has \(\textbf{P}\). 
	\end{corollary}
	
	\noindent We prove the theorem above by reducing from the first property of \ref{intro:dagger} and using the following lemma.
	
	\begin{lemma}[\cref{prop: main 3}]
		Let \(X\) and \(Y\) be schemes. If \(f\colon X\to Y \) is finite locally free, \(\textbf{P}\) a property that satisfies \ref{dagger} and \(Y\) has \(\textbf{P}\), then \(X\) has \(\textbf{P}\).
	\end{lemma}

	\noindent As mentioned above, \cref{section 1} is dedicated towards some local topological properties that are of interest. The main argument can be found in \cref{subsection 2.2} but its proofs, specifically those of \cref{prop: main 3} and \cref{prop: main 4}, require some prior findings. These findings make up a proper subset of Grothendieck's works but even so may not be well-known to everyone. Hence, we have included an appendix recalling these findings.
	
	\subsubsection*{Acknowledgements}
	
	\thanks{ 
		I would first like to express my gratitude to Prof. Dr. Torsten Wedhorn not only for providing me with the topic but also for the many helpful remarks and suggestions towards understanding this topic.\vspace{0.15cm}\\
		Second, I want to thank Catrin Mair for her advice and remarks on this paper.
		\vspace{0.15cm}\\
		Third, my thanks go to my friends for often spending their lunch breaks listening to my ramblings and discussing various aspects of my thesis, be it norms, group actions or possible counterexamples e.g. for \cref{prop: main 2} when I was still uncertain whether the statement is true or not. In particular, I want to thank Julian, Michelle, Niklas, Ronald and Saskia for their helpful remarks on this thesis.\vspace{0.15cm}\\
		Finally, I would like to express special thanks to Johan de Jong for providing the idea of a proof of \cref{prop: main 2} and to Ofer Gabber for his remarks on this proposition as well as his explanations relating to \cref{prop: main 3 helper}. \vspace{0.15cm}
	}
	
	\numberwithin{theorem}{section}
	
	\section{Local topological properties}\label{section 1}
	We first recall some basic results regarding the properties we are interested in.
	
	\begin{def+prop}\cite[I \textsection 11 Definition 4]{Bou95}\label{def: top 1}
		A topological space \(X\) is called \emph{locally connected} if one of the following two equivalent conditions are satisfied:
		\begin{enumerate}[label=(\roman*)]
			\item Every point of \(X\) has a fundamental system of connected neighborhoods,
			\item for any open subset \(U\subseteq X\), the connected components of \(U\) are open.
		\end{enumerate}
		
	\end{def+prop}
	
	\begin{lemma}\cite[Lemma A.1]{gabber} \label{lemma: top 1}
		If $f : X \rightarrow  Y$ is an open continuous map and if $X$ is locally connected, then $f(X)$ is locally connected. In particular, $X$ is locally connected as soon as it has a covering by locally connected open subspaces.
	\end{lemma} 
	
	\begin{definition}\cite[Exercise 3.15]{GW1} \label{def: top 2}
		We say that a topological space \(X\) has \emph{locally finitely many irreducible components} if every point of \(X\) has an open neighborhood which meets only finitely many irreducible components of \(X\).
	\end{definition}
	
	\begin{lemma}\label{lemma: top 2}
		Let $X$ be a topological space and $U\subseteq X$ an open set. Then 
		\begin{align*}
			\left\{W\subseteq U \text{ irred. comp.}\right\} &\xleftrightarrow{\text{1:1}} \left\{\substack{V \subseteq X \text{ irred. comp.} \\ U\cap V \neq \emptyset }\right\}\\
			W &\mapsto \overline{W}\\
			V\cap U &\mapsfrom V  
		\end{align*}
		is a bijection of sets.
	\end{lemma}
	
	\begin{lemma}\label{lemma: top 3}
		A topological space $X$ with locally finitely many irreducible components locally has finitely many connected components and is therefore locally connected.
	\end{lemma}
	\begin{proof}
		By \cref{lemma: top 1} it suffices to show local connectedness on an open cover. Hence, we consider an open cover $(U_i)_{i\in I}$ such that each $U_i$ has only finitely many irreducible components and prove that the connected components of some open $Q\subseteq U_i$ are open. As the irreducible components of $Q$ are connected and cover $Q$, we see that they must in particular cover all connected components of $Q$. As \(Q\) is open, it has only finitely many irreducible components. Therefore, we conclude that there are also only finitely many connected components and hence, they are open.
	\end{proof}
	
	\begin{lemma}\label{lemma: top 4}
		Let \(X\) be a quasi-compact topological space which has locally  finitely many irreducible components. Then \(X\) has finitely many irreducible components.
	\end{lemma}
	\begin{proof}
		Let \((U_i)_{i\in I}\) be an open cover of \(X\) such that \(U_i\) has only finitely many irreducible components for all \(i\in I\). As \(X\) is quasi-compact, we can take a finite subcover \((U_{i})_{i=1}^{n}\). Due to \cref{lemma: top 2}, we conclude that \(X\) has only finitely many irreducible components.
	\end{proof}

	\begin{definition}\cite[\href{https://stacks.math.columbia.edu/tag/0051}{Tag 0051}]{stacks}\label{def: top 3}
		A topological space \(X\) is called \emph{Noetherian} if the descending chain condition holds for closed subsets of \(X\), i.e. every descending chain of closed subsets \(Z_1\supseteq Z_2\supseteq \dots\) becomes stationary. A topological space is called \emph{locally Noetherian} if every point has an open neighborhood which is Noetherian.
	\end{definition}
	
	\begin{definition}\cite[Definition 6.6.9]{Bhatt Scholze}
		A scheme \(X\) is said to be \emph{topologically Noetherian (resp. locally topologically Noetherian)} if its underlying topological space is Noetherian (resp. locally Noetherian).
	\end{definition}
	
	\begin{lemma}\label{lemma: top 5}
		Let \(X\) be a topological space. Then the following are equivalent:
		\begin{enumerate}[label=(\roman*)]
			\item \(X\) is Noetherian. \label{top_noe:item1}
			\item Every ascending chain of open subsets \(U_1\subseteq U_2\subseteq \dots\) of \(X\) becomes stationary.\label{top_noe:item2}
			\item Every open subset \(U\subseteq X\) is quasi-compact.\label{top_noe:item3}
		\end{enumerate}
	\end{lemma}
	
	\begin{lemma}\label{lemma: top 6}
		Let \(X\) be a quasi-compact locally Noetherian topological space. Then \(X\) is Noetherian.
	\end{lemma}
	\begin{proof}
		Let \((U_i)_{i\in I}\) be an open cover of \(X\) such that \(U_i\) is Noetherian for all \(i\in I\). As \(X\) is quasi-compact, we can take a finite subcover \((U_{i})_{i=1}^{n}\). Due to \cite[\href{https://stacks.math.columbia.edu/tag/0053}{Tag 0053}]{stacks}, we conclude that \(X\) is Noetherian.
	\end{proof}
	
	\begin{lemma}
		Every locally Noetherian topological space \(X\) has locally finitely many irreducible components.
	\end{lemma}
	\begin{proof}
		This is an immediate consequence of \cite{stacks} \href{https://stacks.math.columbia.edu/tag/0052}{Tag 0052} (3).
	\end{proof}
	
	\noindent Finally, we note that locally Noetherian schemes have an underlying locally topologically Noetherian topological space. So overall, we see that for a scheme
	\begin{align*}
		\substack{ \text{locally}\\\text{Noetherian} }\Rightarrow
		\substack{\text{locally topologically}\\\text{ Noetherian}}\Rightarrow \substack{\text{locally finitely many}\\\text{ irreducible components} }\Rightarrow \substack{ \text{locally}\\\text{ connected} }
	\end{align*}
	and for a topological space
	\begin{align*}
		\substack{ \text{locally topologically}\\\text{ Noetherian} }\Rightarrow \substack{ \text{locally finitely many}\\\text{ irreducible components} }\Rightarrow \substack{ \text{locally}\\\text{ connected.} }
	\end{align*}

	\section{Permanence properties of schemes under preimages} \label{section 2}
	\noindent Let \(\textbf{P}\) be a property of schemes and let the following four conditions be denoted by \textlabel{$\dagger$}{dagger}: 
	\begin{itemize}
		\item Let \(f\colon X\to Y\) be an integral surjective morphism of affine schemes, and $G$ the constant \(Y\)-group scheme of a finite group acting on \(X\) via \(\rho \colon G\times_YX\to X\), such that \(\rho\times_Yp_X\colon G\times_Y X\to X\times_YX\) is surjective. Then, if \(Y\) has \(\textbf{P}\), so does \(X\).
		\item Let \(A,B\) be rings. If \(\spec{\left(B\right)}\to \spec{\left(A\right)}\) is a surjective morphism of affine schemes and \(\spec{\left(B\right)}\) has \(\textbf{P}\), then so does \(\spec{\left(A\right)}\).
		\item If a scheme \(X\) has \(\textbf{P}\), then so does every open subscheme \(U\subseteq X\).
		\item If a scheme \(X\) has an open covering \((X_i)_{i\in I}\) such that each \(X_i\) has \(\textbf{P}\), then \(X\) has \(\textbf{P}\).
	\end{itemize}
	
	\noindent The main result of this paper will show that such a property is  stable under taking locally quasi-finite, flat and locally finitely presented preimages.

	\noindent Before we come to the proof, we first consider some examples of properties that fulfill $\dagger$.

	\subsection{Examples of properties stable under flat, locally quasi-finite and locally finitely presented preimages}\phantom{a}\\ 
	\noindent The three examples we give in this section are:
	\begin{itemize}
		\item The scheme \(X\) is locally connected.
		\item The scheme \(X\) has locally finitely many irreducible components.
		\item The scheme \(X\) is locally topologically Noetherian.
	\end{itemize}
	
	\noindent In order to prove that these properties fulfill \ref{dagger}, however, we first need to introduce two lemmas and give a definition to avoid confusion.
	
	\begin{definition}\cite[\href{https://stacks.math.columbia.edu/tag/0406}{Tag 0406}]{stacks}
		Let \(f\colon X\to Y\) be a continuous map of topological spaces.
		\begin{enumerate}[label=(\roman*)]
			\item We say \(f\) is a \emph{strict map} of topological spaces if the induced topology and the quotient topology of \(f(X)\) agree.
			\item We say \(f\) is \emph{submersive} if \(f\) is surjective and strict.
		\end{enumerate}
	\end{definition}
	
	We now give two lemmas that generalize parts of \cite[Proposition A.3]{gabber}:
	\begin{lemma}\label{prop: main 1}
		If $f\colon X\to Y$ is a submersive morphism of schemes, and $G$ a constant \(Y\)-group scheme acting on \(X\) via \(\rho \colon G\times_YX\to X\), such that the morphism 
		\begin{align*}
			\rho\times_Yp_X\colon G\times_Y X\to X\times _Y X
		\end{align*}
		is surjective, then \(f\) is open.
	\end{lemma}
	\begin{proof}
		Choose any open $ U\subseteq X$. We know that the group action is transitive on the fibers of $f$. In other words, for any $y\in Y$ and any $x,z\in f^{-1}(y)$ there exists some $g\in G$ such that $\rho (g,x)=z$. Now, $\bigcup_{g\in G}^{}\rho (g,U)\subseteq f^{-1}(f(U))$ and conversely, let $z\in f^{-1}(f(U))$ be arbitrary and $x\in U\cap f^{-1}(f(z))\ne \emptyset$. Then there exists some $g\in G$ such that $\rho (g,x)=z$. Hence, it follows that 
		\begin{align}\label{invariant}
			f^{-1}(f(U))=\bigcup_{g\in G}^{}\rho (g,U).
		\end{align}
		As automorphisms preserve openness and unions of open sets are again open, we see that $f^{-1}(f(U))$ is open. Hence, as \(f\) is submersive, \(f(U)\) is open.
	\end{proof}
	
	\begin{lemma}
		If $f\colon X\to Y$ is a submersive morphism of schemes, and $G$ the constant \(Y\)-group scheme of a finite group acting on \(X\) via \(\rho \colon G\times_YX\to X\), such that the morphism 
		\begin{align*}
			\rho\times_Yp_X\colon G\times_Y X\to X\times _Y X
		\end{align*}
		is surjective, then \(f\) is closed.
	\end{lemma}
	\begin{proof}
		As automorphisms also preserve closedness, we see that the proof is entirely analogous to the proof of \cref{prop: main 1}.
	\end{proof}
	
	\noindent These two lemmas lead to the following proposition by Gabber:
	
	\begin{prop}\cite[Proposition A.3]{gabber}
		Let $f\colon X\to Y$ be an integral, surjective morphism of schemes, and $G$ the constant \(Y\)-group scheme of a finite group acting on \(X\) via \(\rho \colon G\times_YX\to X\), such that the morphism 
		\begin{align*}
			\rho\times_Yp_X\colon G\times_Y X\to X\times _Y X
		\end{align*}
		is surjective. Then, if $Y$ is locally connected, so is $X$.
	\end{prop}
	
	\begin{prop}
		Let \(X\to Y\) be a surjective, quasi-compact morphism of schemes. Then, if \(X\) is locally connected, so is \(Y\).
	\end{prop}
	\begin{proof}
		As the quasi-compact open subsets form a basis of the topology, it suffices to show the claim for quasi-compact opens. So, let \(V\subseteq Y\) be a quasi-compact open. Its inverse image \(U\coloneqq f^{-1}(V)\) in \(X\) is locally connected. Hence, its connected components are open. As \(U\) is also quasi-compact, we deduce that it only has finitely many connected components. Now, because continuous images of connected sets are connected, we see that the images of the connected components of \(U\) are connected and cover \(V\). Hence, \(V\) must have only finitely many connected components.
	\end{proof}
	
	\begin{corollary}\label{cor: main 1}
		The property of being a locally connected scheme satisfies \ref{dagger}.
	\end{corollary}
	\begin{proof}
		This follows from the propositions above, as well as \cref{def: top 1} and \cref{lemma: top 1}.
	\end{proof}

	\begin{prop} 
		Let $f\colon X\to Y$ be a submersive morphism of schemes, and $G$ the constant \(Y\)-group scheme of a finite group acting on \(X\) via \(\rho \colon G\times_YX\to X\), such that the morphism 
		\begin{align*}
			\rho\times_Yp_X\colon G\times_Y X\to X\times _Y X
		\end{align*}
		is surjective. Then, if $Y$  has locally finitely many irreducible components, so does $X$.
	\end{prop}
	\begin{proof}
		Denote by \(n\) the cardinality of the finite group associated to \(G\). We reduce to the case where \(Y=\spec{\left(B\right)}\) has finitely many irreducible components. Then for any fiber \(X_y\) of \(f\), we know that \( \card(X_y)\leq  n <\infty\). We consider the set of generic points \(P\coloneqq\left\{\xi_1,\dots,\xi_r\right\}\) of \(Y\) and notice that because \(f\) is open (cf. \cref{prop: main 1}), \(f^{-1}(P)\) is dense in \(X\). Further, it is finite and because taking the closure commutes with unions, we conclude by \cite{stacks} \href{https://stacks.math.columbia.edu/tag/0G2Y}{Tag 0G2Y} after removing potentially unnecessary points from \(f^{-1}(P)\). 
	\end{proof}
	
	\begin{prop}
		Let \(X\to Y\) be a surjective, quasi-compact morphism of schemes. Then, if \(X\) is quasi-compact and has locally finitely many irreducible components, so does \(Y\).
	\end{prop}
	\begin{proof}
		We may reduce to the case where \(X\) has only finitely many irreducible components by \cref{lemma: top 4}. Now, as images of irreducible sets are irreducible, we see that the images of the irreducible components of \(X\) are a finite covering of irreducible sets. Hence, \(Y\) must have finitely many irreducible components. 
	\end{proof}
	
	\begin{corollary}\label{cor: main 2}
		The property of being a scheme with locally finitely many irreducible components satisfies \ref{dagger}. 
	\end{corollary}
	\begin{proof}
		This follows from the propositions above, as well as \cref{def: top 2} and the fact that integral, surjective morphisms are in particular submersive.
	\end{proof}

	\noindent We are grateful to Johan de Jong for providing the idea for the following proof.
	
	\begin{prop}\label{prop: main 2}
		Let $f : X \to Y$ be a continuous map of sober topological spaces which is closed, has finite fibers, and has the following property: for all $y \in Y$ the group $\operatorname{Aut}(X/Y)$ acts transitively on the fiber $f^{-1}(y)$. If $Y$ is Noetherian, then $X$ is Noetherian.
	\end{prop}
	
	\begin{proof}
		First, note that if \(V\subseteq Y\) is closed, then \(f\vert_{f^{-1}(V)}\colon f^{-1}(V)\to V\) also has all assumptions made about \(f\). As such, by Noetherian induction, we may assume that for all proper closed subsets \(V\subset Y\), the preimage \(f^{-1}(V)\) is Noetherian. If \(Y=\emptyset\), then \(X\) is  also empty and the claim holds true. If \(Y\) is a reducible, non-empty topological space, then let \(Y=V_1\cup V_2\) be a covering via proper, closed subsets \(V_1,V_2\subseteq Y\). If the \(V_i\) are Noetherian, then so are their preimages \(f^{-1}(V_i)\) and therefore, \(X=f^{-1}(V_1)\cup f^{-1}(V_2)\) is Noetherian. This leaves us with the case that \(Y\) is  irreducible. Let \(\xi\in Y\) be its generic point and let \(f^{-1}(\xi)=\{\eta_1,\ldots,\eta_r\}\). Further, denote for all \(j=1,\ldots,r\) the closure \(X_j\coloneqq \overline{\{\eta_j\}}\subseteq X\). Then the \(X_j\) are irreducible components of \(X\), each with image \(Y\) under \(f\). If \(r=0\), then \(f\) is not surjective and \(X=f^{-1}(f(X))\) is the preimage of  a proper, closed subset of \(Y\), hence Noetherian. We may now assume that \(r\ge 1\). For \(g\in \mathrm{Aut}(X/Y)\), we find that \(g(X_1)=X_j\) for some \(j\) and therefore, \(X=X_1\cup\ldots\cup X_r\) because \(f\) is closed which means specializations lift along \(f\).\\
		Further, we note that \(\eta_i\not\in X_j\) for \(i\ne j\in \{1,\ldots,r\}\) due to sobriety. This is important because otherwise, the argument below fails\footnote{In particular,  we would need to remove some indices in the list of irreducible components and, more importantly, we could no longer justify \(\eta_j\not\in T_n\) for \(1\le j\le s\).}.\\
		Let \(X\supseteq Z_1\supseteq Z_2\supseteq \ldots\) be a decreasing sequence of closed subsets. Showing that this sequence stabilizes finishes the proof. We consider the sequence of \(\{\eta_1,\ldots,\eta_r\}\supseteq Z_1\cap \{\eta_1,\ldots,\eta_r\}\supseteq Z_2\cap \{\eta_1,\ldots,\eta_r\}\supseteq\ldots \) which becomes stationary eventually due to finiteness of its components. After relabelling the \(\eta_j\), we find a \(0\le s\le r\) and an \(n\in \mathbb{N}\) such that for all \(i\ge n\), we have \(Z_i\supseteq X_1\cup\ldots\cup X_s\) and \(Z_i\not\supseteq X_j\) for \(r\ge j\ge s+1\). Now, we define \(T_i\) as the closure of \(Z_i\setminus X_1\cup\ldots\cup X_s\) and note that the following two properties hold: 
		\begin{enumerate}
			\item \(Z_i=X_1\cup\ldots\cup X_s\cup T_i\).
			\item \(f^{-1}(f(T_n))\supseteq T_n\supseteq T_{n+1}\supseteq \ldots\) is a decreasing sequence of closed subsets.
		\end{enumerate}
		As such, if we show that \(f(T_n)\) is a proper closed subset of \(Y\), the sequence of \(T_i\) stabilizes and then, so does the sequence of \(Z_i\).
		Because \(Z_n\setminus X_1\cup\ldots\cup X_s\) is a subset of the closed set \(X_{s+1}\cup\ldots\cup X_r\), we see that \(T_n\subseteq X_{s+1}\cup\ldots\cup X_r\). Thus, \(\eta_j\not\in T_n\)  for \(j=1,\ldots,s\). Also, for \(s+1\le j\le r\), by construction \(Z_n\cap X_j\subsetneq X_j\) is a proper closed subset so it cannot contain \(\eta_j\). Therefore, \(\eta_j\not\in T_n\) for all \(j=1,\ldots,r\) and hence, \(\xi\not \in f(T_n)\) showing that \(f(T_n)\) is a proper closed subset. Now, Noetherian induction yields the claim.
	\end{proof}

	\begin{corollary}
		If $f\colon X\to Y$ is an integral surjective morphism of schemes, and $G$ the constant \(Y\)-group scheme of a finite group acting on \(X\) via \(\rho \colon G\times_YX\to X\), such that the morphism 
		\begin{align*}
			\rho\times_Yp_X\colon G\times_Y X\to X\times _Y X
		\end{align*}
		is surjective. Then, if $Y$ is locally topologically Noetherian, so is $X$.
	\end{corollary}
	\begin{proof}
		This follows immediately from \cref{prop: gr_act,prop: main 2}.
	\end{proof}
	\begin{prop}
		Let \(\spec{\left(B\right)}\to \spec{\left(A\right)}\) be a surjective morphism of affine schemes. Then if \(\spec{\left(B\right)}\) is locally topologically Noetherian, so is \(\spec{\left(A\right)}\).
	\end{prop}
	\begin{proof}
		Due to quasi-compactness of \(\spec{\left(B\right)}\) we may again reduce to the case where \(\spec{\left(B\right)}\) is topologically Noetherian by \cref{lemma: top 6}. Now, let \(U\subseteq \spec{\left(A\right)}\) be an open subset and \((U_i)_{i\in I}\) an open covering of \(U\). The preimage of \(U\) is quasi compact as \(\spec{\left(B\right)}\) is topologically Noetherian and hence, we can take a finite subcover of \((f^{-1}(U_i))_{i\in I}\), say \((f^{-1}(U_1),\dots,f^{-1}(U_n))\) (possibly after relabelling). As  \(f\) is surjective, we have \(f(f^{-1}(V))=V\) for all \(V\subseteq \spec{\left(A\right)}\) and hence, \((f(f^{-1}(U_1)),\dots, f(f^{-1}(U_n)))=(U_1,\dots,U_n)\) is a finite subcover, proving that \(U\) is quasi-compact. Hence, by \cref{lemma: top 5}  \(\spec{\left(A\right)}\) is topologically Noetherian.
	\end{proof}
	
	\begin{corollary}\label{cor: main 3}
		The property of being a locally topologically Noetherian scheme satisfies \ref{dagger}. 
	\end{corollary}
	\begin{proof}
		This follows from the statements above, as well as \cref{def: top 3}.
	\end{proof}

	\subsection{Reducing to flat, locally quasi-finite morphisms locally of finite presentation}\label{subsection 2.2}
	\noindent Following the proofs of \cite{gabber}, we are able to generalize most of its results  to \(\textbf{P}\) as demonstrated in this section.
	\begin{lemma}\label{prop: main 3}
		Let \(X\) and \(Y\) be schemes. If \(f\colon X\to Y \) is finite locally free, \(\textbf{P}\) a property that satisfies \ref{dagger} and \(Y\) has \(\textbf{P}\), then \(X\) has \(\textbf{P}\).
	\end{lemma}
	\begin{proof}
		We may reduce to the affine case \(X=\spec{\left(B\right)}\to \spec{\left(A\right)}=Y\), where \(B\) is locally free of rank \(d> 0\) over \(A\) as \(f\) is affine, and because we can check for \(\textbf{P}\) locally in the target. Let \(C\coloneqq \otimes _A^dB\) and denote by \(\nu\) the morphism associated with the norm \(\Gamma ^{d}_A B=C^{\mathfrak{S}_d} \to A\), in the sense that \(\norm_{B/A}(b)=\nu(\gamma (b))\), where \(\gamma (b)=b^{\otimes d}\in (\otimes _A^{d}B)^{\mathfrak{S}_d}\), and let \(G\) be the constant \(Y\)-group scheme associated to \(\mathfrak{S}_d\). Consider the co-cartesian (=tensor product) square of rings and the corresponding cartesian square of spectra
		\begin{center}
			\begin{tikzcd}
				C^{\mathfrak{S}_d}\arrow[r]{}{}\arrow[d,two heads]{}{\nu}&C\arrow[d,two heads]{}{q}&\arrow[r,mapsto]{}{\spec{}}&\phantom{a}& \spec{\left(C^{\mathfrak{S}_d}\right)}&\spec{\left(C\right)}\arrow[l,two heads]{}{}\\
				A\arrow[r]{}{}&D\arrow[ul, phantom, "\scalebox{1.2}{\color{black}$\ulcorner$}", very near start, color=black]&\phantom{a}&\arrow[l,mapsto,swap]{}{\Gamma }&\spec{\left(A\right)}\arrow[u,hook, swap]{}{\spec{\left(\nu\right)}}&\spec{\left(D\right)}.\arrow[u,hook, swap]{}{\spec{\left(q\right)}}\arrow[l,two heads]{}{}	\arrow[ul, phantom, "\scalebox{1.2}{\color{black}$\ulcorner$}", very near start, color=black]
			\end{tikzcd}
		\end{center}
		We thus have \(D=(\otimes _A^{d}B)/I(\otimes _A^{d}B)\), where \(I=\operatorname{Ker}\left(\nu\right)\) due to \(\nu\) being surjective (cf. \cref{prop: norms 1}), and \(G\) is acting on \(\spec{\left(D\right)}\) in such a way that \(G\times_Y \spec{\left(D\right)}\to \spec{\left(D\right)}\times_{\spec{\left(A\right)}}\spec{\left(D\right)}\) is surjective because \(\spec{\left(A\right)}\) is a quotient scheme of \(\spec{\left(D\right)}\) by \(\mathfrak{S}_d\). The first property of \ref{dagger} then yields that  \(\spec{\left(D\right)}\) has \(\textbf{P}\). \\
		In the cocartesian square, the top horizontal map is an integral extension of rings, so in particular also injective and hence, surjective on \(\spec\). So because surjectivity is stable under base change, the bottom horizontal arrow is surjective on \(\spec\) and thus, the morphism \(A\to D\) has nil kernel. \\
		For \(1 \leq i \leq d\), let \(q_{i}\colon B\to D\) be the map \(b\mapsto q(1 \otimes \dots 1\otimes b \otimes 1 \dots \otimes 1)\), with \(b\) at the \(i\)-th place. Let \(b\in B\) be such that \(q_1(b)=0\). We will show that \(b\) is nilpotent. The characteristic polynomial \(P(T)=\charpol_{B/A}(b,T)\in A[T]\) of \(b\) factors, in \(D[T]\), as \(\prod (T-q_i(b))\) by \cref{prop: main 3 helper}. 
		Since \(\mathfrak{S}_d\) permutes the \(q_i\) for all \(i\), one has \(q_i(b)=0\), and the image of \(P(T)\) in \(D[T]\) reduces to \(T^{d}\), showing that the coefficients of \(P(T)\) of degree \(<d\) are nilpotent. By Cayley-Hamilton we deduce that \(b\) is nilpotent. 
		
		\noindent Therefore, the morphism \(\spec{\left(q_1\right)}\colon \smash{\spec{\left(D\right)} \to \spec{\left(B\right)}}\) is a dominant morphism of affine schemes. It is even surjective because it is finite and thus, closed. Finiteness follows because as \(A\to B\) is finite, so is \(B\to \otimes ^{d}_AB\), the morphism \(\otimes ^{d}_AB\to D\) is surjective, so finite and composition of finite morphisms is finite. By the second condition of \ref{dagger}, if \(\spec{\left(D\right)}\) has \(\textbf{P}\), so does \(\spec{\left(B\right)}\), finishing the proof.
	\end{proof}

	\begin{theorem}\label{prop: main 4}
		Let $f\colon X \to Y$ be a flat, locally quasi-finite morphism of schemes locally of finite presentation, \(\textbf{P}\) a property that satisfies \ref{dagger} and let $Y$ have \(\textbf{P}\). Then $X$ has \(\textbf{P}\).
	\end{theorem}
	\begin{proof}
		We first show that this holds if \(f\) is étale and deduce the main result from there.
		If \(f\) is étale, we can apply the local structure theorem for étale morphisms (see \cite[18.4.6 (ii)]{EGAIV4} or \cite[\href{https://stacks.math.columbia.edu/tag/025A}{Tag 025A}]{stacks}): Let \(x\in X\) be a point and \(V\subseteq Y\) an affine neighborhood of \(f(x)\). We find an affine open \(x\in U\subset X\) with \(f(U)\subseteq V\) such that the diagram 
		\begin{center}
			\begin{tikzcd}
				X\arrow[d]{}{f}&U\arrow[l]{}{}\arrow[d]{}{f\vert _U}\arrow[r]{}{j}&\spec{\left(R[t]_{g'}/(g)\right)}\arrow[d]{}{}\\
				Y&\arrow[l]{}{}V\arrow[r,equal]{}{}&\spec{\left(R\right)},
			\end{tikzcd}
		\end{center}
		where \(j\) is an open immersion, commutes. Further, we notice that \(R[t]_{g'}/(g)\cong (R[t]/(g))_{g'}\) and that \(R[t]/(g)\) is a free \(R\) module of rank \(\mathrm{deg}(g)\) (recall that polynomial division with remainder is well-defined). Hence, \(\spec{\left(R[t]_{g'}/(g)\right)}\to \spec{\left(R[t]/(g)\right)}\) is an open immersion and \(\spec{\left(R[t]/(g)\right)}\to R\) is finite and locally free so, \(\spec{\left(R[t]/(g)\right)}\) has \(\textbf{P}\) by \cref{prop: main 3}. Then \ref{dagger} tells us that any open subset has \(\textbf{P}\) and because \(j\) is an open immersion, we conclude that \(U\) has \(\textbf{P}\). We can therefore cover \(X\) with open subschemes that have \(\textbf{P}\) and conclude by the fourth property of \ref{dagger} that \(X\) has \(\textbf{P}\). \\
		For the main result, we use \cite[18.12.1]{EGAIV4}:
		For each (isolated) $x\in X$ there exists an étale morphism $g\colon Y' \to Y$ and a commutative diagram, in which we denote $X'\coloneqq X\times_Y Y'$
		\vspace{-0.25cm}
		\begin{center}
			\begin{tikzcd}
				V\arrow[r]{}{j}\arrow[rd]{}{}&X'\arrow[dr, phantom, "\scalebox{1.2}{\color{black}$\lrcorner$}" , very near start, color=black]\arrow[r]{}{g'}\arrow[d]{}{f'}&X\arrow[d]{}{f}\\
				&Y'\arrow[r]{}{g}&Y,
			\end{tikzcd}
		\end{center}
		where the morphism $j$ is an open immersion and thus $g'(j(V))$ is an open set containing $x$ and $f'\circ j$ is finite. Hence, $V\to Y'$ is a finite, flat morphism locally of finite presentation. From above we deduce that $Y'$ has \(\textbf{P}\) and due to \cite[\href{https://stacks.math.columbia.edu/tag/02KB}{Tag 02KB}]{stacks} we can apply \cref{prop: main 3} and conclude that $V$ has \(\textbf{P}\). We have thus found an open covering of $X$ such that all the elements in the covering have \(\textbf{P}\), proving $X$ to have \(\textbf{P}\).
	\end{proof}
	
	\noindent In \cite{gabber}, Gabber further proves that for \(\textbf{P}=\text{"locally connected"}\) we can even reduce to morphisms that are flat and locally of finite presentation. 
	
	\section*{Appendix}
	\appendix
	
	\section{Symmetric Tensors and homogeneous polynomials}
	\noindent This section, as well as the following section, serves to give some insight into the proof of \cref{prop: main 3} and to recall some lesser known properties. \\
	More specifically, this section will concern homogeneous polynomials \(\underline{M}\to \underline{N}\) of degree \(n>0\), the corresponding functor of sets of homogeneous polynomials from \(M\) to \(N\), and its representability by a module \(\Gamma^{n} (M)\). We will further investigate how \(\Gamma ^{n}(M)\) is related to \(\mathrm{TS}^{n}(M)\), the module of symmetric tensors.
	\begin{notation}
		For this section, let \(A\) be a ring,  \(M\) an \(A\)-module and denote by \(\underline{M}\) the functor \(A\text{-}\text{Alg}\to \text{Set}, B\mapsto M\otimes_AB\). 
	\end{notation}
	
	\begin{theorem}\cite[I 1.2]{lazard}\label{lazards_theorem}
		Let \(A\) be a ring. An \(A\)-module \(M\) is flat if and only if it is a filtered colimit of free, finitely generated \(A\)-modules.  
	\end{theorem}
	
	\noindent This theorem, while being a strong result in itself, yields the following corollaries which will be of importance for our constructions.
	
	\begin{corollary}\cite[Exp. XVII 5.5.1.2]{SGA43}\label{SGA 5.5.1.2}
		The functor \(\colim{}{}\) is an equivalence of the category of Ind-objects of free, finitely generated \(A\)-modules and the category of flat \(A\)-modules.
	\end{corollary} 
	
	\begin{corollary}\cite[I 1.5]{lazard}\label{SGA 5.5.1.3}
		Let \(A,B\) be two rings. Then any functor \(T\colon A\text{-}\mathrm{Mod}\to B\text{-}\mathrm{Mod}\) which maps free, finitely generated modules to flat modules and commutes with filtered colimits, maps flat modules to flat modules.
	\end{corollary}
	
	\noindent With this, we can give the relevant definitions.
	
	\begin{definition}\cite[Exp. XVII 5.5.2.1]{SGA43}
		Let \(M\) be an \(A\)-module and \(\mathfrak{S}_n\) the permutation group of \(n>0\) elements. Then \(\mathfrak{S}_n\) acts on \(\otimes_A^nM \) by permuting the entries of the elements\footnote{In other words, for \(\sigma\in \mathfrak{S}_n \) we have \(\sigma(m_1\otimes\dots\otimes m_n)=m_{\sigma (1)}\otimes\dots\otimes m_{\sigma (n)}\).}. We denote
		\begin{align*}
			\mathrm{TS}^{n}(M)=\mathrm{TS}_A^n(M)\coloneqq (\otimes _A^{n}M)^{\mathfrak{S}_n}
		\end{align*}
		and call it the \emph{module of symmetric tensors (of degree \(n\))}.
	\end{definition} 
	
	\begin{remark}
		If \(B\) is an \(A\)-algebra, then \(\mathrm{TS}^{n}(B)\) is a sub-\(A\)-algebra of \(\otimes ^{n}_AB\).
	\end{remark}
	
	\begin{definition}\cite[Exp. XVII 5.5.2.2]{SGA43}\label{SGA 5.5.2.2}
		Let \(A\) be a ring and \(M, N\) be \(A\)-modules. A \emph{polynomial map} \(f\) from \(M\) to \(N\) is a morphism of functors \(\underline{M}\to \underline{N}\). A polynomial map \(f\colon \underline{M}\to \underline{N}\) is said to be \emph{homogeneous of degree} \(n\) if for any \(A\)-algebra \(B\), any element \(\lambda \in B\) and any element \(m\in \underline{M}(B)\), we have \(f(\lambda m)=\lambda ^{n}f(m)\). Homogeneous polynomials of degree \(n\) are also called \(n^{\mathrm{fold}}\) \emph{maps} and we denote the set of homogeneous polynomials of degree \(n\) from \(M\) to \(N\) by
		\begin{align*}
			\Hom n^{\mathrm{fold}}{\left(M,N\right)}.
		\end{align*}
		More generally, if \((M_i)_{i\in I}\) is a finite family of \(A\)-modules and \(\bm{n}=(n_i)_{i\in I}\) is a family of integers, a \emph{multihomogeneous polynomial map of multidegree} \(\bm{n}\) from the modules \(M_i\) to \(N\) is a polynomial \(f\colon \underline{\bigoplus_{i\in I} M_i}\to \underline{N}\) for which it again holds that for any base change \(A\to A'\) and elements \(a_i\in A', x_i\in M_i\), we have 
		\begin{align*}
			f_{A'}((a_ix_i)_{i\in I})=\prod_{i\in I}a_i^{n_i}f_{A'}((x_i)_{i\in I}).
		\end{align*}
		We denote the set of multihomogeneous polynomials of degree \(\bm{n}\) from \((M_i)_{i\in I}\) to \(N\) by \[\Hom\bm{n}^{\mathrm{fold}}{\left((M_i)_{i\in I},N\right)}.\]
	\end{definition}
	
	\begin{def+prop}\cite[IV 1]{roby}\cite[Exp. XVII 5.5.2.3]{SGA43}\label{SGA 5.5.2.3}
		Let \(A\) be a ring and \(M\) be an \(A\)-module. The covariant functor
		\begin{align*}
			N\mapsto \Hom n^{\mathrm{fold}}{\left(M,N\right)}
		\end{align*}
		is representable by an \(A\)-module, denoted by \(\Gamma ^{n}(M)\), \(\Gamma ^{n}_A(M)\), or simply \(\Gamma ^{n}M\) and is equipped with a map \(\gamma ^{n}\colon M\to \Gamma ^{n}(M)\), called the \emph{universal} \(n^{\mathrm{fold}}\) \emph{(polynomial) map}, which is homogeneous of degree \(n\) with the universal property that for any \(f\in \Hom n^{\mathrm{fold}}{\left(M,N\right)}\), there is a linear map \(\tilde{f}\colon \Gamma ^{n}(M)\to N\) such that \(f=\tilde{f}\circ \gamma ^{n}\). Given \((M_i)_{i\in I}\) and \(\bm{n}\) as in \cref{SGA 5.5.2.2}, the covariant functor
		\begin{align*}
			N\mapsto \Hom\bm{n}^{\mathrm{fold}}{\left((M_i)_{i\in I},N\right)}
		\end{align*}
		is representable by the module \(\otimes _{i\in I} \Gamma ^{n_i}(M_i)\), and is equipped with a map \((m_i)_{i\in I}\mapsto \otimes _{i\in I}\gamma ^{n_i}(m_i)\) called \emph{universal multihomogeneous polynomial map (of multidegree} \(\bm{n}\)\emph{)}.	
	\end{def+prop}
	\begin{proof}
		See \cite{roby} chapters III and IV.
	\end{proof}
	
	\begin{remark}\cite[Exp. XVII 5.5.2.4]{SGA43}\label{SGA 5.5.2.4}
		Both \(\mathrm{TS}^{d}\) and \(\Gamma ^{n}\) commute with filtered colimits (for the latter, see \cite[Proposition IV.6]{roby}). Further, they map free modules to free modules and hence, by \cref{SGA 5.5.1.2}, flat modules to flat modules.
	\end{remark}
	
	\begin{def+prop}\cite[Exp. XVII 5.5.2.5]{SGA43}\label{SGA43 5.5.2.5}
		The \(n^{\mathrm{fold}}\) map \(M\mapsto \mathrm{TS}^{n}(M),\) \( m\mapsto m^{\otimes n}\) defines an \(A\)-linear map \(\varphi \colon \Gamma ^{n}(M)\to \mathrm{TS}^{n}(M)\), called the \emph{canonical morphism}, for which \(\varphi (\gamma ^{n}(m))=m^{\otimes n}\). If \(M\) is a free, finitely generated module, then \(\varphi \) is an isomorphism, and hence, this is also the case if \(M\) is flat.
	\end{def+prop}
	\begin{proof}
		In the case where \(M\) is free, see \cite[Proposition IV.5]{roby}. Then, the case where \(M\) is flat follows by \cref{lazards_theorem} and \cref{SGA 5.5.2.4}.
	\end{proof}
	
	\begin{prop}\cite[Exp. XVII 5.5.2.8]{SGA43}\label{SGA 5.5.2.8}
		Let \(A\) be a ring and \(M_1,M_2\) \(A\)-modules. If \(u\colon M_1\times M_2\to N\) is a bilinear map, then \(\gamma ^{n}\circ u\) is a bihomogeneous polynomial map of bidegree \((n,n)\) and thus, defines by \cref{SGA 5.5.2.3} a map \(u^{n}\colon \Gamma ^{n}(M_1)\otimes_A\Gamma ^{n}(M_2)\to \Gamma ^{n}(N)\) for which \(u^{n}(\gamma ^{n}(x)\otimes \gamma ^{n}(y))=\gamma ^{n}(u(x,y))\) for all \(x\in M_1, y\in M_2\). In particular, if \(B\) is an \(A\)-algebra, the product \(B\times B\to B\) induces a product on \(\Gamma ^{n}(B)\). This product gives \(\Gamma ^{n}(B)\) an \(A\)-algebra structure, and the canonical map \(\Gamma ^{n}(B)\to \mathrm{TS}^{n}(B)\) is a homomorphism of \(A\)-algebras.
	\end{prop}

	\section{Norms as homogeneous polynomials}
	\begin{definition}\cite[Chapter 3 \textsection 9.1 Definition 1]{bourbaki}
		Let \(A\) be a ring, \(B\) an \(A\)-algebra and \(M\ne 0\) be a \(B\)-module such that, as an \(A\)-module, \(M\) is locally free of finite rank. For all \(b\in B\), denote the \(A\)-linear map \(m_{b}\colon M\to M, x\mapsto b\cdot x\). The trace, determinant and characteristic polynomial of \(m_b\) are called the \emph{trace}, \emph{norm} and \emph{characteristic polynomial of} \(b\) \emph{relative to} \(B\) and are denoted by \(\trace_{M/A}(b)\), \(\norm_{M/A}(b)\) and \(\charpol_{M/A}(b;T)\) respectively.
	\end{definition}
	
	\begin{lemma}\cite[Chapter 3 \textsection 9.1 p.542]{bourbaki}
		Let \(A\) be a ring, \(A'\) and \(B\) two \(A\)-algebras and \(M\ne 0\) a \(B\)-module such that, as an \(A\)-module, \(M\) is locally free of finite rank. Further, let \(M'=M\otimes_AA'\) and \(B'=B\otimes_AA' \) giving \(M'\) a \(B'\)-module structure. Let \(b\in B\), then 
		\begin{enumerate}[label=(\roman*)]
			\item \(\trace_{M'/A'}(b\otimes 1) =\trace_{M/A}(b)\cdot 1\),
			\item \(\norm_{M'/A'}(b\otimes 1) =\norm_{M/A}(b)\cdot 1\),
			\item \(\charpol_{M'/A'}(b\otimes 1;T) =\charpol_{M/A}(b;T)\cdot 1\).
		\end{enumerate}
		In particular, if we take \(A'=A[T]\) and \(M=B\), then 
		\begin{align*}
			\charpol_{B/A}(b;T)=\norm_{B[T]/A[T]}(T-b).
		\end{align*}
	\end{lemma}
	
	\noindent With the definition out of the way, we turn our attention toward establishing a connection between norms and homogeneous polynomials, and hence, by \cref{SGA43 5.5.2.5}, between norms and symmetric tensors.
	
	\begin{remark}\cite[Exp. XVII 6.3.1]{SGA43}\label{SGA43 6.3.1}
		Let \(A\) be a ring. If \(L\) and \(M\) are locally free \(A\)-modules of finite rank, the map \(\Hom_{A}{\left(L,M\right)}\to \Hom_{A}{\left(\bigwedge ^{m}L,\bigwedge ^{m}M\right)}\) is compatible with base change and \(m^{\mathrm{fold}}\) (cf. \cref{SGA 5.5.2.2}). By \cref{SGA 5.5.2.3}, it defines a morphism
		\(\bigwedge ^{m}\colon \Gamma ^{m}\Hom_{A}{\left(L,M\right)}\to \Hom_{A}{\left(\bigwedge ^{m}L,\bigwedge ^{m}M\right)}\)
		such that, for \(A'\) an \(A\)-algebra and  \(u\in \Hom_{A'}{\left(L\otimes_AA',M\otimes_AA'\right)}\), we have 
		\begin{align*}
			\bigwedge ^{m}(\gamma ^{m}(u))=\bigwedge ^{m}(u).
		\end{align*}
		As \(\Hom_{A}{\left(L,M\right)}\) is a vector bundle and hence flat, by \cref{SGA43 5.5.2.5} this morphism is identified with a morphism \(\bigwedge ^{m}\colon \mathrm{TS}^{m}\Hom_{A}{\left(L,M\right)}\to \Hom_{A}{\left(\bigwedge ^{m}L,\bigwedge ^{m}M\right)}\).
	\end{remark}
	
	\begin{corollary}
		Let \(A\) be a ring and \(B\) be a locally free \(A\)-algebra of finite rank \(d>0\). The maps \(\operatorname{det}_{\underline{B}/\underline{A}}\colon \underline{\mathrm{End}_A(B)}\to \underline{A}\) and \(\norm_{\underline{B}/\underline{A}}\colon \underline{B}\to \underline{A},\) given respectively by \(\operatorname{det}_{\underline{B}(A')/\underline{A}(A')}\) and \(\norm_{\underline{B}(A')/\underline{A}(A')}\) for \(A'\) an \(A\)-algebra, are homogeneous polynomials of degree \(d\).
	\end{corollary}
	\begin{proof}
		This follows immediately from \cref{SGA43 6.3.1} choosing \(M=L=B\) and \(m=d\) and noting that \(\norm_{\underline{B}/\underline{A}}\) is simply a restriction \(\operatorname{det}_{\underline{B}/\underline{A}}\) to morphisms \(\underline{B}(A')\to \underline{B}(A'), x\mapsto bx\) for \(A'\) an \(A\)-algebra and \(b\in \underline{B}(A')\).
	\end{proof}
	
	\begin{prop}\cite[Exp. XVII 6.3.1]{SGA43}\label{SGA43 6.3.1 helper}
		Let \(A\) be a ring and let \(L,M,N\) be locally free modules of finite rank. Further, let \(u\in \Hom_{A}{\left(L,M\right)}, v\in \Hom_{A}{\left(M,N\right)}\). Then \(\bigwedge^m(u\circ v)=\bigwedge^m (u)\circ \bigwedge^m(v)\).
	\end{prop}
	\begin{proof}
		We find that the diagram
		\begin{center}
			\begin{tikzcd}
				\Gamma ^{m}\Hom_{A}{\left(L,M\right)}\otimes_A \Gamma ^{m}\Hom_{A}{\left(M,N\right)}\arrow[r]{}{}\arrow[d]{}{\bigwedge^{m}\otimes \bigwedge ^{m}}& \Gamma ^{m}\Hom_{A}{\left(L,N\right)}\arrow[d]{}{\bigwedge ^{m}}\\
				\Hom_{A}{\left(\bigwedge ^{m}L,\bigwedge ^{m}M\right)}\otimes_A \Hom_{A}{\left(\bigwedge ^{m}M,\bigwedge ^{m}N\right)}\arrow[r]{}{} &\Hom_{A}{\left(\bigwedge ^{m}L,\bigwedge ^{m}N\right)},
			\end{tikzcd}
		\end{center}
		where the horizontal arrows are given by composition (and \cref{SGA 5.5.2.8} for the top horizontal map) is commutative by \cref{SGA43 6.3.1}. 
		
		\noindent Let by \(c\) the bilinear composition map \[c\colon \Hom_{A}{\left(L,M\right)}\times \Hom_{A}{\left(M,N\right)}\to \Hom_{A}{\left(L,N\right)}.\] By \cref{SGA 5.5.2.8} this induces the linear map \(c^{m}\) for which we have \(c^{m}(\gamma ^{m}(u)\otimes \gamma ^{m}(v))=\gamma ^{m}(c(u,v))\). Further, let \[c'\colon \Hom_{A}{\left(\bigwedge ^{m}L,\bigwedge ^{m}M\right)}\otimes_A \Hom_{A}{\left(\bigwedge ^{m}M,\bigwedge ^{m}N\right)}\to \Hom_{A}{\left(\bigwedge ^{m}L,\bigwedge ^{m}N\right)}\] be the linear map induced by composition. Then with \cref{SGA43 6.3.1} we find 
		\begin{align*}
			\bigwedge ^{m} (c^{m}(\gamma ^{m}(u)\otimes \gamma ^{m}(v)))&=\bigwedge ^{m}(\gamma ^{m}(c(u,v)))	=\bigwedge ^{m}(u\circ v)
			\intertext{and}
			c'\left(\bigwedge ^{m}(\gamma ^{m}(u))\otimes \bigwedge ^{m}(\gamma ^{m}(v))\right)&=\bigwedge ^{m}(u)\circ \bigwedge ^{m}(v).  
		\end{align*}
		Hence, we obtain the formula \(\bigwedge ^{m}(u\circ v)=\bigwedge ^{m}(u)\circ \bigwedge ^{m}(v)\).
	\end{proof}
	
	\begin{corollary}\cite[Exp. XVII 6.3.1]{SGA43}\label{cor: SGA 6.3.1}
		Let \(A\to B\) be a locally free algebra of finite rank \(d>0\). As the norm is a homogeneous polynomial, we obtain an \(A\)-linear map 
		\begin{align*}
			\nu \colon \Gamma ^{d}B=\mathrm{TS}^{d}(B)\to A,\, b^{\otimes d}\mapsto \norm_{B/A}(b).
		\end{align*}
		Then \(\nu\) is not just \(A\)-linear but even an A-algebra homomorphism.
	\end{corollary}
	\begin{proof}
		We apply \cref{SGA43 6.3.1 helper} to the following situation: We set \(L=M=N=B\) and \(m=d\) and note that the norm is simply the determinant restricted to endomorphisms of the form \(x\mapsto b\cdot x\) for \(b\in B\) to obtain the desired result.
	\end{proof}

	\begin{prop}\label{surjectivity}\label{prop: norms 1}
		Let \(A\to B\) be a locally free algebra of finite rank \(d>0\). The map \(\nu\colon \mathrm{TS}^{d}(B)\to A\), defined in \cref{cor: SGA 6.3.1}, is surjective.
	\end{prop}
	\begin{proof}
		We see that \(\nu(1_B^{\otimes d})=1_A\) and conclude by \(A\)-linearity.
	\end{proof}
	
	\noindent Before we can end this section, we must first give a lemma which may seem out of place but will be needed for the subsequent proposition which will be of great importance for \cref{prop: main 3}.

	\begin{prop}\label{prop: norms 2}
		Let \(A\to B\) be a locally free algebra of finite rank, \(d>0\) an integer and \(\mathfrak{S}_d\) be the permutation group of \(d\) elements. Then 
		\begin{align*}
			\eta \colon (\otimes _{A[T]}^{d}B[T])^{\mathfrak{S}_d} &\to \mathrm{TS}^{d}(B)[T]=(\otimes _{A}^{d}B)^{\mathfrak{S}_d}[T],\\
			\left(\sum_{j=0}^{n_1}b_j^{1} T^{j}\right)\otimes \dots\otimes \left(\sum_{j=0}^{n_d}b_j^{d} T^{j}\right) &\mapsto \sum_{\beta \colon \left\{1,\dots,d\right\}\to \mathbb{N}}^{} T^{\sum_{k=1}^{d}\beta (k)} \stackrel[i=1]{d }{\otimes} b_{\beta (i)}^{i}
		\end{align*}
		is an isomorphism.
	\end{prop}
	\begin{proof}
		We note that \(A[T]\) is a flat \(A\)-algebra and that 
		\begin{align*}
			\otimes ^{d}_{A[T]}B[T]\cong \otimes ^{d}_{A[T]} (B\otimes_AA[T])\cong (\otimes_A ^{d}B)\otimes_A A[T],
		\end{align*}
		so the result follows from \cite[Remark 12.28]{GW1}. Further, we note that one can check that the given morphism is indeed the canonical isomorphism.
	\end{proof}

	\begin{prop}\label{prop: main 3 helper}
		Let \(A\to B\) be a locally free algebra of finite rank \(d>0\) and consider the co-cartesian (=tensor product) square
		\begin{center}
			\begin{tikzcd}
				(\otimes _A^{d}B)^{\mathfrak{S}_d}\arrow[r]{}{}\arrow[d,two heads]{}{\nu}&\otimes _A^{d}B\arrow[d,two heads]{}{q}\\
				A\arrow[r]{}{}&D\arrow[ul, phantom, "\scalebox{1.2}{\color{black}$\ulcorner$}", very near start, color=black]
			\end{tikzcd}
		\end{center}
		where \(\nu\colon TS^{d}(B)=(\otimes _A^{d}B)^{\mathfrak{S}_d}\to A\) is defined by \(\nu (b^{\otimes d})=\norm_{B/A}(b)\) (cf. \cref{cor: SGA 6.3.1}). Moreover, define 
		\begin{align*}
			q_i\colon B&\to D\\
			b&\mapsto q(1\otimes\dots 1\otimes b\otimes 1\dots\otimes 1)
		\end{align*} where \(b\) is in the \(i\)-th place. Then \(\charpol_{B/A}(b;T)\) factors as \(\prod_{i=1}^{d}(T-q_i(b))\) in \(D[T]\).
	\end{prop}
	\begin{proof}
		Let \(\nu_{A[T]}\colon (\otimes ^{d}_{A[T]}B[T])^{\mathfrak{S}_{d}}\to A[T]\) be the linear map associated to \(\norm_{B[T]/A[T]}\) and let \(\nu'\colon (\otimes _A^{d}B)^{\mathfrak{S}_d}[T]\to A[T]\) be the map where \(\nu\) is applied to the coefficients. Then \(\nu_{A[T]}=\nu'\circ \eta\).
		As \(\nu_{A[T]}((T-b)^{\otimes d})=\norm_{B[T]/A[T]}(T-b)=\charpol_{B/A}(b;T)\), it hence suffices to show that \((T-b)^{\otimes d}\in (\otimes _{A[T]}^{d}B[T])^{\mathfrak{S}_d}\cong (\otimes _A^{d}B)^{\mathfrak{S}_d}[T]\) maps to \(\prod_{i=1}^{d}(T-b_i)\in (\otimes _A^{d}B)[T]\), where \(b_i=1\otimes \dotsm 1\otimes b\otimes 1\dotsm \otimes 1\) with \(b\) in the \(i\)-th place, via \(\eta \) defined in \cref{prop: norms 2}. We see that \(\eta ((T-b)^{\otimes d})= T^{d}-\mathfrak{s}_1(b_1,\dots, b_d)T^{d-1}+\dots+(-1)^{d}\mathfrak{s}_d(b_1,\dots, b_d)\) where \(\mathfrak{s}_i\) is the \(i\)-th elementary symmetric polynomial. This factors as \(T^{d}-\mathfrak{s}_1(b_1,\dots, b_d)T^{d-1}+\dots+(-1)^{d}\mathfrak{s}_d(b_1,\dots, b_d)=\prod_{i=1}^{d}(T-b_i)\in (\otimes ^{d}_AB)[T]\) as desired, finishing the proof.
	\end{proof}

	\section{Group actions on schemes} 
	
	\begin{def+prop}\cite[I.5 Definition 6]{bourbaki}
		Let \(X\) be a set and \(G\) be a group. An action \(\rho \colon G\times X\to X\) of $G$ on $X$ is called \emph{transitive} if the following three equivalent conditions are satisfied:
		\begin{enumerate}[label=(\roman*)]
			\item For any $a\in X$ it holds that $\orbit(a)=X$,
			\item for any $x,z\in X$ there exists some $g\in G$ such that $\rho (g,x)=z$,
			\item the shear map \(\operatorname{sh}\colon G\times X\to X\times X,(g,x)\mapsto (\rho (g,x),x) \) is surjective. 
		\end{enumerate}
	\end{def+prop}
	
	\begin{remark}
		The third notion mentioned above generalizes nicely to arbitrary categories once a decision on what "surjective" means is made. We cannot expect the objects of a general category to have a set-structure so simply being surjective on the underlying sets is not a valid condition in many cases. As we are interested in schemes, however, this will be the definition we use, more concretely:
	\end{remark}

	\begin{definition}
		Let \(S\) be a scheme, \(X\) be an \(S\)-scheme and \(G\) an \(S\)-group scheme with a group action \(\rho \colon G\times_ S X\to X\). Then \(G \) is said to \emph{act transitively} on \(X\)  if the map \(G\times_S X\to X\times_S X, (g,x)\mapsto (\rho (g,x),x)\) is surjective on the underlying topological spaces.
	\end{definition}

	\begin{prop}\label{prop: gr_act}
		Let \(S\) be a scheme, \(X\) be an \(S\)-scheme with structure morphism \(f\colon X\to S\) and \(G\) an \(S\)-group scheme with an action \(\rho \colon G\times_SX\to X\). Then \(\zeta \colon G\times_SX\to X\times_SX\), induced by the projection to \(X\) and \(\rho \), is surjective if and only if for all \(s\in S\), \(G_s\) acts transitively on the corresponding fiber of \(f\).
	\end{prop}
	\begin{proof}
		This follows from the fact that surjectivity can be checked on fibers.
	\end{proof}

\end{document}